\documentclass[journal, letterpaper]{IEEEtran}
\usepackage{amsmath,amsfonts,amssymb,euscript, graphicx, 
epsfig,
enumerate,float,afterpage, subfigure, ifthen}%

\newtheorem{thm}{Theorem}

\newtheorem{lem}{Lemma}

\newcommand{\expect}[1]{\mathbb{E}\left[#1\right]}
\newcommand{\defequiv}{\mbox{\raisebox{-.3ex}{$\overset{\vartriangle}{=}$}}}

\newcommand{\norm}[1]{||{#1}||}

\newcommand{\script}[1]{{{\cal{#1} }}}

\begin{document}

\title
  {A Simple Convergence Time Analysis of Drift-Plus-Penalty for Stochastic Optimization and Convex Programs}
\author{Michael J. Neely\\University of Southern California\\$\vspace{-.5in}$
\thanks{The author is with the  Electrical Engineering department at the University
of Southern California, Los Angeles, CA.} 
\thanks{This work is supported in part  by the NSF Career grant CCF-0747525.}
}

\markboth{}{Neely}

\maketitle

\begin{abstract}   
This paper considers the problem of minimizing the time average of a stochastic process subject to time average constraints on other processes.  A canonical example is minimizing average power in a data network subject to multi-user throughput constraints.  Another example is a (static) convex program. Under a Slater condition, the drift-plus-penalty algorithm is known to provide an $O(\epsilon)$ approximation to optimality with a convergence time of $O(1/\epsilon^2)$.  This paper proves the same result with a simpler technique and in a more general context that does not require the Slater condition. This paper also emphasizes application to basic convex programs, linear programs, and distributed optimization problems.
\end{abstract} 

\section{Introduction}

Fix $K$ as a positive integer.  
Consider a discrete time system that operates over time slots $t \in \{0, 1, 2, \ldots\}$.   Every slot $t$, the controller observes a \emph{random event} $\omega(t)$.  Assume that events $\omega(t)$ are elements in an abstract set $\Omega$, and that they are independent and identically distributed (i.i.d.) over slots.  The set $\Omega$ can have arbitrary (possibly uncountably infinite) cardinality. Every slot $t$, a system controller observes the current $\omega(t)$ and then chooses a \emph{decision vector} $y(t) = (y_0(t), y_1(t), \ldots, y_K(t))$ within an option set $\script{Y}(\omega(t)) \subseteq \mathbb{R}^{K+1}$ that possibly depends on $\omega(t)$.  That is, $\script{Y}(\omega(t))$ is the set of vector options available under the random event $\omega(t)$.  The sets $\script{Y}(\omega(t))$ are arbitrary and are only assumed to have a mild boundedness property (specified in Section \ref{section:algorithm}). 

The goal is to minimize the expected time average of the resulting $y_0(t)$ process subject to time average constraints on the $y_k(t)$ processes for $k \in \{1, \ldots, K\}$. 
Specifically, for integers $t>0$, and for each $k \in \{0, 1, \ldots, K\}$, define: 
\[ \overline{y}_k(t) \defequiv \frac{1}{t}\sum_{\tau=0}^{t-1} \expect{y_k(\tau)} \]
Let $c_1, \ldots, c_K$ be a given collection of real numbers. 
The goal is to solve the following stochastic optimization problem: 
\begin{eqnarray}
\mbox{Minimize:} & \limsup_{t\rightarrow\infty} \overline{y}_0(t) \label{eq:p1}  \\
\mbox{Subject to:} & \limsup_{t\rightarrow\infty} \overline{y}_k(t) \leq c_k \label{eq:p2} \\
& y(t) \in \script{Y}(\omega(t)) \: \: \forall t \in \{0, 1, 2, \ldots\} \label{eq:p3}
\end{eqnarray} 

Assume the problem is \emph{feasible}, so that it is possible to satisfy the constraints 
\eqref{eq:p2}-\eqref{eq:p3}.  Define $y_0^{opt}$ as the infimum value of the objective \eqref{eq:p1} over all algorithms that satisfy the constraints \eqref{eq:p2}-\eqref{eq:p3}.
The \emph{drift-plus-penalty} algorithm from \cite{sno-text} is known to satisfy constraints \eqref{eq:p2}-\eqref{eq:p3} and to ensure: 
\begin{equation} \label{eq:ta-1} 
\limsup_{t\rightarrow\infty} \overline{y}_0(t) \leq y_0^{opt} + \epsilon 
\end{equation} 
where $\epsilon>0$ is a parameter used in the algorithm.  This is done by defining \emph{virtual queues} $Q_k(t)$ for each constraint $k \in \{1, \ldots, K\}$ in  \eqref{eq:p2}: 
\begin{equation} \label{eq:q-update0} 
Q_k(t+1) = \max[Q_k(t) + y_k(t) - c_k, 0] 
\end{equation} 
where $y_k(t)$ acts as a \emph{virtual arrival process} and $c_k$ acts as a constant \emph{virtual service rate}.\footnote{In an actual queueing system, arrivals and service rates are always non-negative.  However, in this virtual queue, the $y_k(t)$ and $c_k$ values can possibly be negative.}  
The intuition behind \eqref{eq:q-update0} is that if $Q_k(t)$ is stable, the time average arrival rate must be less than or equal to the time average service rate, which implies the desired time average constraint \eqref{eq:p2}. 
Under an additional \emph{Slater condition}, it is also known that the drift-plus-penalty algorithm provides an $O(1/\epsilon)$ bound on the time average expected size of all virtual queues: 
\begin{equation} \label{eq:ta-2} 
 \limsup_{t\rightarrow\infty}\overline{Q}_k(t) \leq O(1/\epsilon) \: \: \forall k  \in \{1, \ldots, K\} 
 \end{equation} 
where $\overline{Q}_k(t)$ is defined for $t>0$ by: 
 \[ \overline{Q}_k(t) \defequiv \frac{1}{t}\sum_{\tau=0}^{t-1} \expect{Q_k(\tau)} \]
More recently, it was shown that the \emph{convergence time} required for the desired time averages to ``kick in'' is $O(1/\epsilon^2)$, provided that the Slater condition still holds (see Appendix C in \cite{dist-opt-arxiv}).  Specifically, an algorithm is said to produce an $O(\epsilon)$ approximation with convergence time $T$ if for all $t \geq T$ one has: 
\begin{eqnarray}
\overline{y}_0(t) &\leq& y_0^{opt} + O(\epsilon) \label{eq:conv1}  \\
\overline{y}_k(t) &\leq& c_k + O(\epsilon) \: \: \forall k \in \{1, \ldots, K\} \label{eq:conv2} 
\end{eqnarray}

\subsection{Contributions of the current paper} 

The current paper focuses on the issue of convergence time.  The main result is a proof that convergence time is $O(1/\epsilon^2)$ for general problems that have an associated \emph{Lagrange multiplier}.  It can be shown that a Lagrange multiplier exists whenever the Slater condition exists, but not vice-versa.  Hence the proof in the current paper is more general than the prior result \cite{dist-opt-arxiv} that uses a Slater condition.  To appreciate this distinction, note that a Slater condition is equivalent to assuming there exists a value $\delta>0$ and a decision policy under which \emph{all constraints can be satisfied with at least  $\delta$ slackness}:
\[ \limsup_{t\rightarrow\infty} \overline{y}_k(t) \leq c_k - \delta  \: \: \forall k \in \{1, \ldots, K\} \]
This Slater condition is impossible in many problems of interest.  For example, a problem with a time average equality constraint $\lim_{t\rightarrow\infty} \overline{x}(t) = c$ can be treated using two inequality constraints of the type \eqref{eq:p2}: 
\begin{eqnarray*}
\limsup_{t\rightarrow\infty} \overline{x}(t) &\leq& c \\
\limsup_{t\rightarrow\infty} [-\overline{x}(t)] &\leq& -c 
\end{eqnarray*}
However, it is  impossible for a Slater condition to exist with the above two inequality constraints. Indeed, that would require: 
\begin{eqnarray}
\limsup_{t\rightarrow\infty} \overline{x}(t) &\leq& c - \delta \label{eq:pc1}  \\
\limsup_{t\rightarrow\infty}[ -\overline{x}(t)] &\leq& -c - \delta \label{eq:pc2} 
\end{eqnarray}
Yet, \eqref{eq:pc2} implies: 
\begin{eqnarray*}
c + \delta &\leq& -\limsup_{t\rightarrow\infty} [-\overline{x}(t)]  \nonumber \\
&=& \liminf_{t\rightarrow\infty} \overline{x}(t) \nonumber \\
&\leq& \limsup_{t\rightarrow\infty} \overline{x}(t) \nonumber \\
 &\leq& c - \delta 
  \end{eqnarray*} 
 where the final inequality  follows from \eqref{eq:pc1}. This means that $c+\delta \leq c-\delta$, 
a contradiction when $\delta >0$. 

Another contribution of the current paper is the application of this stochastic result to standard (static) convex programs and linear programs.  Of course, static problems are a special case of stochastic problems. Nevertheless, this paper clearly illustrates that point, and shows that the drift-plus-penalty algorithm can be applied to convex programs and linear programs to produce an $\epsilon$-approximation with convergence time $O(1/\epsilon^2)$.  This was previously shown in \cite{neely-dist-comp} under a Slater condition.  A collection of simplified example problems of distributed optimization, similar to those presented in \cite{neely-dist-comp}, are given to demonstrate the method. 

\subsection{Applications} 

The problem \eqref{eq:p1}-\eqref{eq:p3} is useful in a variety of settings, including problems of stochastic network utility maximization \cite{neely-fairness-ton}\cite{neely-thesis}\cite{now} and problems of minimizing average power in a network subject to queue stability \cite{neely-energy-it}.  Indeed, the drift-plus-penalty technique was developed in \cite{neely-fairness-ton}\cite{neely-thesis}\cite{now}\cite{neely-energy-it} for use in these particular applications.  

As an example, consider a multi-user wireless downlink problem where random data arrivals $a_k(t)$ arrive to the base station every slot $t$, intended for different users $k \in \{1, \ldots, K\}$.  Suppose the network controller can observe the current channel state vector $S(t) = (S_1(t), \ldots, S_K(t))$, which specifies current conditions on the channel for each user.  The controller also observes 
the vector of new data arrivals $a(t) = (a_1(t), \ldots, a_K(t))$.  Let $\omega(t) = (S(t); a(t))$ be a concatenated vector with this channel and arrival information, and let $\Omega$ be the set of all possible $\omega(t)$ vectors. Let $p(t) = (p_1(t), \ldots, p_K(t))$ be the power used for transmission, chosen within some abstract set $\script{P}$ every slot $t$.   Let $\mu_k(p(t),S(t))$ be a function that specifies the transmission rate on channel $k$ under the power vector $p(t)$ and the channel state vector $S(t)$ \cite{neely-energy-it}. Define $r_k(t) = \mu_k(p(t),S(t))$.  The goal is to minimize total average power expenditure subject to ensuring the average transmission rate for each channel is greater than or equal to the arrival rate: 
\begin{eqnarray*}
\mbox{Minimize:} & \limsup_{t\rightarrow\infty} \sum_{k=1}^K\overline{p}_k(t) \\
\mbox{Subject to:} & \limsup_{t\rightarrow\infty} [\overline{a}_k(t) - \overline{r}_k(t)] \leq 0 \\
& p(t) \in \script{P} \: \: \forall t \in \{0, 1, 2, \ldots\} 
\end{eqnarray*}

This problem has the form \eqref{eq:p1}-\eqref{eq:p3} by defining: 
\begin{eqnarray}
y_0(t) &=& \sum_{k=1}^Kp_k(t) \nonumber \\
y_k(t) &=& a_k(t) - \mu_k(p(t), S(t)) \: \: \forall k \in \{1, \ldots, K\}  \label{eq:yk-example} \\
c_k &=& 0 \: \: \forall k \in \{1, \ldots, K\} \nonumber
\end{eqnarray}
and by defining $\script{Y}(\omega)$ for each $\omega=(S, a)  \in \Omega$ as the set of all $(y_0, y_1, \ldots, y_K) \in \mathbb{R}^{K+1}$ such that there is a vector $p \in \script{P}$ that satisfies: 
\begin{eqnarray*}
y_0 &=& \sum_{k=1}^Kp_k\\
y_k &=& a_k - \mu_k(p,S) \: \: \forall k \in \{1, \ldots, K\} 
\end{eqnarray*}
In this example, the virtual queue equations \eqref{eq:q-update0} reduce to the following for all $k \in \{1, \ldots, K\}$: 
\[ Q_k(t+1) = \max[Q_k(t) + a_k(t) - \mu_k(p(t), S(t)), 0] \]
This ``virtual queue'' corresponds to an \emph{actual network queue}, where $a_k(t)$ is the actual arriving data on slot $t$, and $\mu_k(p(t), S(t))$ is the actual transmission rate offered on slot $t$.

\subsection{Prior work} 

The drift method for queue stability was developed in \cite{tass-radio-nets}\cite{tass-server-allocation}, which resulted in \emph{max-weight} and \emph{backpressure} algorithms for data networks. 
The drift-plus-penalty method was developed for network utility maximization problems in \cite{neely-fairness-ton}\cite{neely-thesis} and energy optimization problems in \cite{neely-energy-it}. 
Generalized tutorial results are in \cite{sno-text}\cite{now}.  The works \cite{sno-text}\cite{now} prove that, under a Slater condition, the drift-plus-penalty algorithm gives an $O(\epsilon)$ approximation to optimality with an average queue size tradeoff of $O(1/\epsilon)$.  Recent work  
in \cite{dist-opt-arxiv} shows that convergence time is  $O(1/\epsilon^2)$ under a Slater condition. 
Application to convex programs are given in \cite{neely-dist-comp}, again under a Slater condition.  

Related work in \cite{atilla-fairness-ton} derives a similar algorithm for utility maximization in a wireless downlink via a different analysis that uses Lagrange multipliers.  Lagrange multiplier analysis was used in \cite{longbo-lagrange-tac} to improve queue bounds to $O(\log(1/\epsilon))$ in certain piecewise linear cases. Work in 
\cite{neely-energy-convergence-arxiv} demonstrates near-optimal convergence time of 
$O(\log(1/\epsilon)/\epsilon)$  for one-link problems with piecewise linearity.  
Improved convergence time bounds of $O(1/\epsilon)$ are recently shown in \cite{sucha-convergence-time} for deterministic problems with piecewise linearity assumptions.  Work in \cite{wei-convergence-admm} considers the special case of a deterministic convex program with linear constraints, and uses a different method for obtaining $O(1/\epsilon)$ convergence time.  The work \cite{wei-convergence-admm} also considers distributed implementation over a graph. 
While the works 
\cite{neely-energy-convergence-arxiv}\cite{sucha-convergence-time}\cite{wei-convergence-admm} demonstrate convergence time that is superior to the $O(1/\epsilon^2)$ result of the current paper, those results hold only for special case systems.

 The drift-plus-penalty algorithm is closely related to the dual subgradient algorithm for convex programs \cite{bertsekas-convex}.   
Related work in \cite{lin-shroff-cdc04} uses a dual subgradient approach for non-stochastic problems of network scheduling for utility maximization.  Network scheduling with stochastic approximation is considered in \cite{lee-stochastic-scheduling}.  A different \emph{primal-dual} approach is 
considered for network utility maximization in \cite{prop-fair-down}\cite{vijay-allerton02}\cite{stolyar-greedy}\cite{stolyar-gpd-gen}\cite{atilla-primal-dual-jsac}. 

\section{Algorithm and basic analysis} \label{section:algorithm} 

This section presents the basic results needed from \cite{sno-text}.

\subsection{Boundedness assumption} \label{section:boundedness} 

Assume there are non-negative constants $h_0, h_1, \ldots, h_K$ such that under any policy for making decisions and for any given slot $t$, the first moment of $y_0(t)$ and the 
second moments of $y_k(t)$ for $k \in \{1, \ldots, K\}$ satisfy: 
\begin{eqnarray}  
\expect{|y_0(t)|}     &\leq&     h_0 \label{eq:hk0} \\
 \expect{y_k(t)^2}   &\leq&    h_k \: \: \forall k \in \{1, \ldots, K\} \label{eq:hk1} 
 \end{eqnarray} 
That is, the first moment of $y_0(t)$ is uniformly bounded for all $t$, and the second moments of 
$y_k(t)$ for $k \in \{1, \ldots, K\}$ are also uniformly bounded.   

These boundedness conditions \eqref{eq:hk0}-\eqref{eq:hk1} are satisfied, for example, whenever there is a bounded set $\script{Y} \subseteq \mathbb{R}^{K+1}$ such that $\script{Y}(\omega) \subseteq \script{Y}$ for all $\omega \in \Omega$.  It can also hold when $y_k(t)$ is not necessarily bounded.  This is useful in the wireless downlink example with $y_k(t)=a_k(t) - \mu_k(p(t), S(t))$, as defined by \eqref{eq:yk-example}.  Suppose that $\mu_k(\cdot)$ always takes values in the bounded interval $[0, r_{max}]$ for some real number $r_{max}>0$.  In this case, $y_k(t)$ satisfies \eqref{eq:hk1} whenever $\expect{a_k(t)^2}$ is finite.
However, particular values of 
$y_k(t)$ can be arbitrarily large if the arrivals $a_k(t)$ can be arbitrarily large.  For example, if $a_k(t)$ is a Poisson process, it can take arbitrarily large values but has a finite second moment.

\subsection{Compactness assumption} \label{section:compactness} 

Assume that for all $\omega \in \Omega$, the set $\script{Y}(\omega)$ is a compact subset of $\mathbb{R}^{K+1}$ (recall that a subset is compact if it is closed and bounded).  This compactness assumption  is not crucial to the analysis,  but it simplifies exposition.\footnote{If $\script{Y}(\omega)$ is not compact, one can still obtain optimality results by assuming the drift-plus-penalty algorithm comes within an additive constant $C$ of minimizing the desired expression for all slots $t$.  This is called a \emph{$C$-additive approximation} \cite{sno-text}.}   Indeed, such compactness 
ensures that, given any $\omega \in \Omega$, there is always an optimal solution to problems of the following type: 
\begin{eqnarray*}
\mbox{Minimize:} & \sum_{k=0}^{K} w_k y_k \\
\mbox{Subject to:} & (y_0, \ldots, y_{K}) \in \script{Y}(\omega) 
\end{eqnarray*}
where $w_0, \ldots, w_{K}$ are a given set of real numbers.  The drift-plus-penalty algorithm will be shown to make decisions every slot $t$ according to such a minimization.

The sets $\script{Y}(\omega(t))$ are not required to have any additional structure beyond the boundedness and compactness assumptions specified in Sections \ref{section:boundedness} and \ref{section:compactness}.  In particular, the sets $\script{Y}(\omega(t))$ might be finite, infinite, convex, or non-convex.

\subsection{The set $\script{R}$ of all average vectors} 

Recall that random events $\omega(t)$ are i.i.d. over slots.  The distribution for $\omega(t)$ is possibly unknown.  Imagine observing $\omega(t)$ and randomly choosing vector $y(t)$ in the set $\script{Y}(\omega(t))$ according to a distribution that depends on $\omega(t)$. The expectation vector 
$\expect{y(t)}$ is with respect to the randomness of $\omega(t)$ and the conditional randomness of $y(t)$ given $\omega(t)$.  Define $\script{R}$ as the set of all expectation vectors $\expect{y(t)} = \expect{(y_0(t), \ldots, y_K(t))}$ that can be achieved, considering all  $\omega \in \Omega$ and all possible conditional distributions over the set $\script{Y}(\omega)$ given that $\omega(t)=\omega$.  A probabilistic mixture of two randomized choices is again a randomized choice, and so the set $\script{R}$ is a convex subset of $\mathbb{R}^{K+1}$.   The boundedness assumptions \eqref{eq:hk0}-\eqref{eq:hk1} further imply that  
$\script{R}$ is bounded.

Every slot $\tau \in \{0, 1, 2, \ldots \}$, a general algorithm chooses $y(\tau)$ as a (possibly random) vector in the set $\script{Y}(\omega(\tau))$ (with distribution that possibly depends on the observed history), and so $\expect{y(\tau)} \in \script{R}$ for all slots $\tau$.  Fix $t>0$. It follows that 
$\overline{y}(t)= \frac{1}{t}\sum_{\tau=0}^{t-1} \expect{y(\tau)}$ is a convex combination of vectors in $\script{R}$, and so it is also in $\script{R}$ (since $\script{R}$ is a convex set).  That is: 
\begin{equation} \label{eq:in-R} 
\overline{y}(t) \in \script{R} \: \: \forall t \in \{1, 2, 3, \ldots\} 
\end{equation} 

\subsection{Optimality} 

Define $\overline{\script{R}}$ as the closure of $\script{R}$.  Since $\script{R}$ is a bounded and convex subset of $\mathbb{R}^{K+1}$, the set $\overline{\script{R}}$ is a compact and convex subset of $\mathbb{R}^{K+1}$. Consider the problem: 
\begin{eqnarray}
\mbox{Minimize:} & y_0 \label{eq:static1}  \\
\mbox{Subject to:} & y_k \leq c_k \: \: \forall k \in \{1, \ldots, K\} \label{eq:static2}  \\
& (y_0, y_1, \ldots, y_K) \in \overline{\script{R}} \label{eq:static3} 
\end{eqnarray}
In \cite{sno-text} it is shown that the above problem  \eqref{eq:static1}-\eqref{eq:static3} is feasible if and only if the original stochastic optimization problem 
\eqref{eq:p1}-\eqref{eq:p3} is feasible.  Further, assuming feasibility, the problems \eqref{eq:static1}-\eqref{eq:static3}  and 
\eqref{eq:p1}-\eqref{eq:p3} have the same optimal objective 
value $y_0^{opt}$. 

Throughout this paper it is assumed that problem \eqref{eq:p1}-\eqref{eq:p3} is feasible, and hence
problem \eqref{eq:static1}-\eqref{eq:static3} is feasible.  Let $(y_0^{opt}, y_1^{opt} ,\ldots, y_K^{opt})$ be an optimal solution to \eqref{eq:static1}-\eqref{eq:static3}. Such an optimal solution exists because the problem 
\eqref{eq:static1}-\eqref{eq:static3} is feasible and the set $\overline{\script{R}}$ is compact.  This optimal solution must satisfy the constraints of problem \eqref{eq:static1}-\eqref{eq:static3}, and so: 
\begin{equation} \label{eq:go-to-zero} 
y_k^{opt} \leq c_k \: \: \forall k \in \{1, \ldots, K\} 
\end{equation}

\subsection{Lyapunov optimization} 

Define $Q(t) = (Q_1(t), \ldots, Q_K(t))$ as the vector of queue backlogs. The squared norm of the backlog vector is:
\[ \norm{Q(t)}^2 = \sum_{k=1}^KQ_k(t)^2 \]
Define $L(t) = \frac{1}{2}\norm{Q(t)}^2$, called a \emph{Lyapunov function}.  The drift-plus-penalty algorithm observes the current vector $Q(t)$ and random event $\omega(t)$ every slot $t$, and then 
makes a decision $y(t) \in \script{Y}(\omega(t))$ to greedily minimize a bound on the \emph{drift-plus-penalty} expression: 
\[ \Delta(t) + V y_0(t) \]
where $V$ is a positive weight that affects a performance tradeoff.  Setting $V = 1/\epsilon$ results in an $O(\epsilon)$ approximation to optimality \cite{sno-text}.  This fact is reviewed in the remainder of this section, as several of the key results are needed in the new convergence analysis of Section \ref{section:convergence}. 

To bound $\Delta(t)$, fix $k \in \{1, \ldots, K\}$, square the queue equation \eqref{eq:q-update0}, and use the fact that $\max[z,0]^2 \leq z^2$ to obtain: 
\[ Q_k(t+1)^2 \leq Q_k(t)^2 + (y_k(t)-c_k)^2 + 2Q_k(t)(y_k(t)-c_k) \]
Summing the above over $k \in \{1, \ldots, K\}$ and dividing by $2$ gives: 
\[ \Delta(t) \leq B(t) +  \sum_{k=1}^KQ_k(t)(y_k(t) - c_k) \]
where $B(t)$ is defined: 
\begin{equation} \label{eq:Bt} 
B(t) = \frac{1}{2}\sum_{k=1}^K (y_k(t) - c_k)^2 
\end{equation} 
Adding $Vy_0(t)$ to both sides gives the following bound: 
\begin{equation} \label{eq:dpp} 
 \Delta(t) + Vy_0(t) \leq B(t) + Vy_0(t) + \sum_{k=1}^KQ_k(t)(y_k(t) - c_k) 
 \end{equation} 
Every slot $t$, the drift-plus-penalty algorithm observes $Q(t), \omega(t)$ and chooses $(y_0(t), y_1(t), \ldots, y_K(t))$ in the set  $\script{Y}(\omega(t))$ to minimize the last two terms on the right-hand-side of 
\eqref{eq:dpp}. 

\subsection{Drift-plus-penalty algorithm} 

Initialize $Q_k(0)=0$ for all $k \in \{1, \ldots, K\}$.  Perform the following steps every slot $t \in \{0, 1, 2, \ldots\}$: 
\begin{itemize} 
\item Observe $Q(t) = (Q_1(t), \ldots, Q_K(t))$ and $\omega(t)$, and choose $(y_0(t), \ldots, y_K(t)) \in \script{Y}(\omega(t))$ to minimize: 
\begin{equation} \label{eq:alg-choice} 
Vy_0(t) + \sum_{k=1}^KQ_k(t)y_k(t) 
\end{equation} 
\item Update queues $Q_k(t)$ for $k \in \{1, \ldots, K\}$ via: 
\begin{equation} \label{eq:q-update} 
Q_k(t+1) = \max[Q_k(t) + y_k(t) - c_k, 0] 
\end{equation} 
\end{itemize} 

A key feature of this algorithm is that it reacts to the observed state $\omega(t)$, and 
does not require knowledge of the probability distribution associated with $\omega(t)$.  Notice that once the queue vector $Q(t)$ is observed on slot $t$, its components act as known weights in the minimization of \eqref{eq:alg-choice}.  Hence, this minimization indeed has the form specified in Section \ref{section:compactness}.  Specifically, every slot a vector $y(t) \in \script{Y}(\omega(t))$ is chosen to minimize a linear function of the components $y_0(t), y_1(t), \ldots, y_K(t)$.  Complexity of this decision
depends on the structure of the sets $\script{Y}(\omega(t))$.  If these sets consist of a finite and small number of points, the decision amounts to testing each option and choosing the one with the least weighted sum.  The decision can be complex if the sets $\script{Y}(\omega(t))$ consist of a finite but large number of points, or if these sets are infinite but non-convex.   

For simplicity, it is assumed throughout that $y(t)$ is chosen to exactly minimize the expression 
\eqref{eq:alg-choice}  (this is possible via the compactness assumption of Section \ref{section:compactness}).  Similar analytical results can be obtained under the weaker assumption 
that $y(t)$ comes  within an additive constant of minimizing \eqref{eq:alg-choice}, 
called a \emph{$C$-additive approximation} (see \cite{sno-text}).

\subsection{Constraint satisfaction via queue stability} 

The queue backlog gives a simple bound on constraint violation.  Indeed, for all slots $\tau \in \{0, 1, 2, \ldots\}$ one has from \eqref{eq:q-update} and the fact that $\max[z, 0] \geq z$: 
\[ Q_k(\tau+1) \geq Q_k(\tau) + y_k(\tau) - c_k \]
Thus: 
\[ Q_k(\tau+1) - Q_k(\tau) \geq y_k(\tau) - c_k \]
Summing over $\tau \in \{0, 1, \ldots, t-1\}$ for some integer $t>0$ gives: 
\[ Q_k(t) - Q_k(0) \geq \sum_{\tau=0}^{t-1} y_k(\tau)  - tc_k \]
Dividing by $t$ and using the fact that $Q_k(0)=0$ gives:  
\[ \frac{Q_k(t)}{t} \geq \frac{1}{t}\sum_{\tau=0}^{t-1} y_k(\tau) - c_k \]
Taking expectations gives: 
\begin{equation*} 
\frac{\expect{Q_k(t)}}{t} \geq \overline{y}_k(t) - c_k 
\end{equation*}
Rearranging terms gives the desired constraint violation bound: 
\begin{equation} \label{eq:violation} 
\overline{y}_k(t) \leq c_k + \frac{\expect{Q_k(t)}}{t} 
\end{equation}  
It follows that the desired constraints \eqref{eq:p2} hold if all queues $k \in \{1, \ldots, K\}$ satisfy: 
\begin{equation} \label{eq:mean-rate-stable} 
 \lim_{t\rightarrow\infty} \frac{\expect{Q_k(t)}}{t} = 0 
 \end{equation} 
A queue that satisfies \eqref{eq:mean-rate-stable}  is said to be \emph{mean rate stable} \cite{sno-text}.

\subsection{Objective function analysis} 

Fix $\tau \in \{0, 1, 2, \ldots\}$. 
Because the drift-plus-penalty decision minimizes the last two terms on the right-hand-side of the drift-plus-penalty bound \eqref{eq:dpp},  one has: 
\begin{eqnarray} 
\Delta(\tau) + Vy_0(\tau) &\leq& B(\tau) + Vy_0^*(\tau) \nonumber \\
&& + \sum_{k=1}^KQ_k(\tau)(y_k^*(\tau) - c_k) \label{eq:the-policy} 
\end{eqnarray} 
for all vectors $(y_0^*(\tau), \ldots, y_K^*(\tau)) \in \script{Y}(\omega(\tau))$, including vectors that are chosen \emph{randomly} over $\script{Y}(\omega(\tau))$.  Fix a vector $(y_0^*, \ldots, y_K^*) \in \script{R}$. Let $y^*(\tau)=(y_0^*(\tau), \ldots, y_K^*(\tau))$ be chosen as a random function of $\omega(t)$ according to a conditional distribution that yields expectation $\expect{y^*(\tau)} = (y_0^*, \ldots, y_K^*)$, but with conditional decisions that are independent of history.  Since $\omega(\tau)$ is itself independent of history, it follows that for all $k \in \{1, \ldots, K\}$, $y_k(\tau)$ is independent of $Q_k(\tau)$, and: 
\begin{equation} \label{eq:product} 
 \expect{y_k(\tau)Q_k(\tau)} = \expect{y_k(\tau)}\expect{Q_k(\tau)} = y_k^*\expect{Q_k(\tau)}  
 \end{equation} 
Taking expectations of \eqref{eq:the-policy} (assuming $y^*(\tau)$ is this randomized policy) 
and substituting \eqref{eq:product} gives: 
\begin{eqnarray}
 \hspace{-.2in} \expect{\Delta(\tau)} + V\expect{y_0(\tau)} &\leq& \expect{B(\tau)} + Vy_0^* \nonumber \\
 && + \sum_{k=1}^K\expect{Q_k(\tau)}(y_k^*-c_k) \label{eq:sub} 
 \end{eqnarray}
 Let $B\geq0$ be a finite constant that satisfies the following for all slots $\tau$: 
 \begin{equation} \label{eq:B} 
 \expect{B(\tau)} \leq B 
 \end{equation} 
 Such a constant $B$ exists by the second moment boundedness assumption \eqref{eq:hk1}. Substituting $B$ into \eqref{eq:sub} gives: 
 \begin{eqnarray*}
 \expect{\Delta(\tau)} + V\expect{y_0(\tau)} &\leq& B + Vy_0^* \nonumber \\
 && + \sum_{k=1}^K\expect{Q_k(\tau)}(y_k^*-c_k) 
 \end{eqnarray*}
 
 The above inequality holds for all $(y_0^*, \ldots, y_K^*) \in \script{R}$.  Take a limit as $(y_0^*, \ldots, y_K^*)$ approaches the point $(y_0^{opt}, \ldots, y_K^{opt}) \in \overline{\script{R}}$ to obtain: 
\[ \expect{\Delta(\tau)} + V\expect{y_0(\tau)} \leq B + Vy_0^{opt} + \sum_{k=1}^K \expect{Q_k(\tau)}(y_k^{opt} - c_k) \]
Substituting \eqref{eq:go-to-zero} into the right-hand-side of the above inequality gives: 
\begin{equation} \label{eq:utility-dpp} 
\expect{\Delta(\tau)} + V\expect{y_0(\tau)} \leq B + Vy_0^{opt} 
\end{equation} 
The inequality \eqref{eq:utility-dpp} holds for all slots $\tau \in \{0, 1, 2, \ldots\}$.   Fix $t>0$.  Summing \eqref{eq:utility-dpp}  over $\tau\in\{0, 1, \ldots, t-1\}$ gives: 
\[ \expect{L(t)} - \expect{L(0)} + V\sum_{\tau=0}^{t-1}\expect{y_0(\tau)} \leq (B + Vy_0^{opt})t \]
Dividing by $t$ and using the fact that $\expect{L(0)}=0$ gives: 
\begin{equation} \label{eq:re-use} 
 \frac{\expect{L(t)}}{t} +  V\overline{y}_0(t) \leq B + Vy_0^{opt} 
 \end{equation} 
Dividing by $V$ and using $\expect{L(t)} \geq 0$ gives: 
\begin{equation} \label{eq:y-0-bound} 
\overline{y}_0(t) \leq y_0^{opt} + B/V 
\end{equation} 
That is, \eqref{eq:y-0-bound} ensures that for all slots $t>0$, the time average expectation 
$\overline{y}_0(t)$ is at most $O(1/V)$ larger than the optimal objective function value $y_0^{opt}$.  Fix $\epsilon>0$.  Using the parameter $V = 1/\epsilon$ gives an $O(\epsilon)$ approximation to optimal utility. 

It remains to show that the desired constraints are also satisfied.  If a Slater assumption holds, it can be shown that queue averages are $O(1/\epsilon)$. The Slater assumption also ensures convergence time is $O(1/\epsilon^2)$ \cite{dist-opt-arxiv}.  The next subsection presents a new analysis to develop $O(1/\epsilon^2)$ convergence time \emph{without} the Slater assumption. 

\section{Convergence time analysis} \label{section:convergence} 

\subsection{Lagrange multipliers} 

Assume the problem \eqref{eq:static1}-\eqref{eq:static3} is feasible.  
Since this problem is convex, a hyperplane in 
$\mathbb{R}^{K+1}$ exists that passes through the point $(y_0^{opt}, c_1, \ldots, c_K)$ and that contains the set $\overline{\script{R}}$ on one side \cite{bertsekas-convex}.  Specifically, there are non-negative values $\gamma_0, \gamma_1, \ldots, \gamma_K$ such that: 
\begin{equation*} 
\gamma_0y_0 + \sum_{k=1}^K \gamma_k y_k \geq \gamma_0y_0^{opt} + \sum_{k=1}^K \gamma_k c_k \: \: \forall (y_0, \ldots, y_K) \in \overline{\script{R}} 
\end{equation*} 
The hyperplane is said to be \emph{non-vertical} if $\gamma_0\neq 0$ \cite{bertsekas-convex}.  If the hyperplane is non-vertical, one can divide the above inequality by $\gamma_0$, define $\mu_k = \gamma_k/\gamma_0$ for all $k \in \{1, \ldots, K\}$, and conclude: 
\begin{equation} \label{eq:exist-lm} 
y_0 + \sum_{k=1}^K \mu_k y_k \geq y_0^{opt} + \sum_{k=1}^K \mu_k c_k \: \: \forall (y_0, \ldots, y_K) \in \overline{\script{R}} 
\end{equation} 
The non-negative vector $(\mu_1, \ldots, \mu_K)$ in \eqref{eq:exist-lm} is called a \emph{Lagrange multiplier vector}.  A Lagrange multiplier vector that satisfies \eqref{eq:exist-lm} exists whenever the separating hyperplane is non-vertical.  It can be shown that the separating hyperplane is non-vertical whenever a Slater condition holds.  Such a non-vertical hyperplane also exists in more general situations without a Slater condition (see ``regularity conditions'' specified in \cite{bertsekas-convex}).   Thus, the assumption that a Lagrange multiplier vector exists is a mild assumption. 

\subsection{Bounding the violations} 

Assume a (non-negative) Lagrange multiplier vector $(\mu_1, \ldots, \mu_K)$ exists so 
that \eqref{eq:exist-lm} holds.  Fix $t>0$.  Recall that \eqref{eq:in-R} ensures $\overline{y}(t) = (\overline{y}_0(t), \ldots, \overline{y}_K(t)) \in \script{R}$.  Since $\script{R} \subseteq\overline{\script{R}}$, by \eqref{eq:exist-lm} one has: 
\[    \overline{y}_0(t) + \sum_{k=1}^K \mu_k \overline{y}_k(t) \geq y_0^{opt} + \sum_{k=1}^K \mu_k c_k \]
Rearranging the above gives: 
\begin{eqnarray} 
y_0^{opt} - \overline{y}_0(t) &\leq& \sum_{k=1}^K \mu_k (\overline{y}_k(t) - c_k) \nonumber \\
&\leq& \sum_{k=1}^K \mu_k \frac{\expect{Q_k(t)}}{t}  \label{eq:here} 
\end{eqnarray}
where the final inequality holds by \eqref{eq:violation}. 

On the other hand, one has by \eqref{eq:re-use}: 
\begin{eqnarray}
\frac{\expect{L(t)}}{t} &\leq& B + V(y_0^{opt} - \overline{y}_0(t)) \nonumber \\
&\leq& B + V\sum_{k=1}^K\mu_k\frac{\expect{Q_k(t)}}{t} \label{eq:here2} \\
&\leq& B + \frac{V}{t}\norm{\mu} \cdot \norm{\expect{Q(t)}} \label{eq:here3} 
\end{eqnarray} 
where \eqref{eq:here2} is obtained by substituting \eqref{eq:here}, and \eqref{eq:here3} is due to the 
fact that the dot product of two vectors is less than or equal to the product of their norms.  Substituting the definition  $L(t) = \frac{1}{2}\norm{Q(t)}^2$ in the left-hand-side of \eqref{eq:here3} gives: 
\[ \frac{1}{2t}\expect{\norm{Q(t)}^2} \leq B + \frac{V}{t} \norm{\mu} \cdot \norm{\expect{Q(t)}} \]
Since $\expect{\norm{Q(t)}^2} \geq \norm{\expect{Q(t)}}^2$, one has: 
\[ \frac{1}{2t} \norm{\expect{Q(t)}}^2 \leq B + \frac{V}{t} \norm{\mu} \cdot \norm{\expect{Q(t)}} \]
Therefore: 
\[ \norm{\expect{Q(t)}}^2 - 2V\norm{\mu} \cdot \norm{\expect{Q(t)}} - 2Bt \leq 0 \]
Define $x = \norm{\expect{Q(t)}}$, $b = -2V\norm{\mu}$, $c=-2Bt$.  Then: 
\begin{equation} \label{eq:quadratic} 
 x^2 + bx + c \leq 0 
 \end{equation} 
The largest value of $x$ that satisfies \eqref{eq:quadratic} is equal to the largest root of the quadratic equation $x^2 + bx + c=0$, and so: 
\[ x \leq \frac{-b + \sqrt{b^2 - 4c}}{2} = V\norm{\mu} + \sqrt{V^2\norm{\mu}^2 + 2Bt} \]
Therefore, for all $t>0$ one has: 
\[ \norm{\expect{Q(t)}} \leq V\norm{\mu} + \sqrt{V^2\norm{\mu}^2 + 2Bt} \]
It follows from \eqref{eq:violation} that for all $k \in \{1, \ldots, K\}$ the constraint violations satisfy: 
\begin{eqnarray}
\overline{y}_k(t) &\leq& c_k + \frac{\expect{Q_k(t)}}{t} \nonumber \\
&\leq& c_k +  \frac{\norm{\expect{Q(t)}}}{t}  \nonumber \\
&\leq& c_k + \frac{V\norm{\mu} + \sqrt{V^2 \norm{\mu}^2 + 2Bt}}{t} \label{eq:follow2} 
\end{eqnarray}
This leads to the following theorem. 

\begin{thm} \label{thm:performance} Fix $\epsilon>0$ and define $V = 1/\epsilon$.  If the problem \eqref{eq:p1}-\eqref{eq:p3} is feasible and the Lagrange multiplier assumption \eqref{eq:exist-lm} holds, then for all $t \geq 1/\epsilon^2$ one has:  
\begin{eqnarray}
\overline{y}_0(t) &\leq& y_0^{opt} + O(\epsilon)  \label{eq:thm1} \\
\overline{y}_k(t) &\leq& c_k + O(\epsilon) \: \: \forall k \in \{1, \ldots, K\} \label{eq:thm2} 
\end{eqnarray}
and so the drift-plus-penalty algorithm with $V=1/\epsilon$ provides an $O(\epsilon)$ approximation with convergence time $O(1/\epsilon^2)$. 
\end{thm} 

\begin{proof} 
Inequality \eqref{eq:thm1} holds from \eqref{eq:y-0-bound} and the fact that $B/V = B\epsilon = O(\epsilon)$. Inequality \eqref{eq:thm2} holds from 
\eqref{eq:follow2} and the fact that: 
\begin{eqnarray*}
\frac{V\norm{\mu} + \sqrt{V^2\norm{\mu}^2 + 2Bt}}{t} &=& \frac{\norm{\mu}}{\epsilon t} + \sqrt{\frac{\norm{\mu}^2}{\epsilon^2 t^2} + \frac{2B}{t}} \\
&\leq& \norm{\mu}\epsilon + \sqrt{\norm{\mu}^2 \epsilon^2 + 2B\epsilon^2} \\
&=& \norm{\mu}\epsilon + \epsilon \sqrt{\norm{\mu}^2  + 2B} \\
&=& O(\epsilon) 
\end{eqnarray*}
\end{proof} 

\section{Equality constraints}  \label{section:equality} 

A similar analysis can be used to treat problems with explicit equality constraints.  Specifically, consider choosing a vector $h(t) = (y_0(t), y_1(t), \ldots, y_K(t), w_1(t), \ldots, w_M(t))$ in a set $\script{H}(\omega(t))$ to solve: 
\begin{eqnarray} 
\mbox{Minimize:} & \limsup_{t\rightarrow\infty} \overline{y}_0(t) \label{eq:equality1} \\
\mbox{Subject to:} & \limsup_{t\rightarrow\infty} \overline{y}_k(t) \leq c_k \: \: \forall k \in \{1, \ldots, K\} \label{eq:equality2} \\
& \lim_{t\rightarrow\infty} \overline{w}_i(t) = d_i \: \: \forall i \in \{1, \ldots, M\} \label{eq:equality3} \\
& h(t) \in \script{H}(\omega(t)) \: \: \forall t \in \{0, 1, 2, \ldots\} \label{eq:equality4} 
\end{eqnarray}
where $c_1, \ldots, c_K$ and $d_1, \ldots, d_M$ are given real numbers. One approach is to change each inequality constraint \eqref{eq:equality3} into two inequality constraints: 
\begin{eqnarray*}
\limsup_{t\rightarrow\infty} \overline{w}_i(t) &\leq& d_i \\
\limsup_{t\rightarrow\infty} [-\overline{w}_i(t)] &\leq& -d_i 
\end{eqnarray*}
This would involve two virtual queues for each $i \in \{1, \ldots, M\}$.  A notationally easier method is to simply change the structure of the virtual queue for equality constraints $i \in \{1, \ldots, M\}$ as follows \cite{sno-text}: 
\begin{equation} \label{eq:z-update} 
Z_i(t+1) = Z_i(t) + w_i(t) - d_i \: \: \forall i \in \{1, \ldots, M\} 
\end{equation}
The inequality constraints \eqref{eq:equality2} have the same virtual queues from before: 
\begin{equation} \label{eq:same-q-update} 
Q_k(t+1) = \max[Q_k(t) + y_k(t) - c_k , 0] 
\end{equation} 
The resulting algorithm is as follows:  Initialize $Z_i(0)=Q_k(0)=0$ for all $i \in \{1, \ldots, M\}$ and $k \in \{1, \ldots, K\}$.  Every slot $t \in \{0, 1, 2, \ldots\}$ do: 
\begin{itemize} 
\item Observe $Q_1(t), \ldots, Q_K(t)$ and $Z_1(t), \ldots, Z_M(t)$ and $\omega(t)$ and choose $h(t) \in \script{H}(\omega(t))$ to minimize: 
\[ Vy_0(t) + \sum_{k=1}^KQ_k(t)y_k(t) + \sum_{i=1}^M Z_i(t)w_i(t) \]

\item Update $Q_k(t)$ for $k \in \{1, \ldots, K\}$ and $Z_i(t)$ for $i \in \{1, \ldots, M\}$ via \eqref{eq:same-q-update} and \eqref{eq:z-update}. 
\end{itemize} 

The analysis of this scenario with equality constraints is similar and is omitted for brevity (see \cite{sno-text}). 

\section{Convex programs} 

Fix $N$ as a positive integer.  
Consider the problem of finding a vector $x=(x_1, \ldots, x_N) \in \mathbb{R}^{N}$ to solve: 
\begin{eqnarray}
\mbox{Minimize:} & f(x) \label{eq:cp1}  \\
\mbox{Subject to:} & g_k(x) \leq c_k \: \: \forall k \in \{1, \ldots, K\} \label{eq:cp2}  \\
& x \in \script{X} \label{eq:cp3} 
\end{eqnarray}
where $\script{X}$ is a convex and compact subset of $\mathbb{R}^N$, functions $f(x)$, $g_1(x), \ldots, g_K(x)$ are continuous and 
convex functions over $x \in \script{X}$, and $c_1, \ldots, c_K$ are given real numbers. 
The problem \eqref{eq:cp1}-\eqref{eq:cp3} is a \emph{convex program}. Assume the problem is feasible, so that there exists a vector that satisfies the constraints \eqref{eq:cp2}-\eqref{eq:cp3}.  The compactness and continuity assumptions ensure there is an optimal solution $x^*\in \script{X}$ that solves the problem 
\eqref{eq:cp1}-\eqref{eq:cp3}.  Define $f^* = f(x^*)$ as the optimal objective function value. 

This convex program is equivalent to a problem of the form \eqref{eq:p1}-\eqref{eq:p3}, and hence can be solved by the drift-plus-penalty method \cite{neely-dist-comp}.  To see this, define $\script{Y}$ as the set of all $(y_0, y_1, \ldots, y_K)$ vectors in $\mathbb{R}^{K+1}$ such that there exists a vector $x \in \script{X}$ that satisfies: 
\begin{eqnarray*}
y_0 &=& f(x) \\
y_k &=& g_k(x) \: \: \forall k \in \{1, \ldots, K\} 
\end{eqnarray*}
Consider a system defined over slots $t \in \{0, 1, 2, \ldots\}$.  Every slot $t$, a controller chooses a vector $x(t) = (x_1(t), \ldots, x_N(t))$ in the (deterministic) set $\script{X}$.  Define: 
\begin{eqnarray*}
y_0(t) &=& f(x(t)) \\
y_k(t) &=& g_k(x(t)) \: \: \forall k \in \{1, \ldots, K\}
\end{eqnarray*}
The goal is to choose $x(t)$ over slots to solve: 
\begin{eqnarray}
\mbox{Minimize:} & \limsup_{t\rightarrow\infty} \overline{y}_0(t) \label{eq:trans1} \\
\mbox{Subject to:} & \limsup_{t\rightarrow\infty} \overline{y}_k(t) \leq c_k \: \: \forall k \in \{1, \ldots, K\} \label{eq:trans2}  \\
& x(t) \in \script{X} \: \: \forall t \in \{0, 1, 2, \ldots\} \label{eq:trans3} 
\end{eqnarray}

\begin{lem} \label{lem:jensen} If $\{x(t)\}_{t=0}^{\infty}$ is a random or deterministic process that satisfies $x(t) \in \script{X}$ for all $t$, then:

a) For all $t>0$, one has $\frac{1}{t}\sum_{\tau=0}^{t-1} x(\tau) \in \script{X}$, and: 
\begin{eqnarray*}
f\left(\frac{1}{t}\sum_{\tau=0}^{t-1} x(\tau)\right) &\leq& \frac{1}{t}\sum_{\tau=0}^{t-1}y_0(\tau) \\
g_k\left(\frac{1}{t}\sum_{\tau=0}^{t-1} x(\tau)\right) &\leq& \frac{1}{t}\sum_{\tau=0}^{t-1} y_k(\tau) \: \: \forall k \in \{1, \ldots, K\} 
\end{eqnarray*}
 
 b) For all $t>0$, $\overline{x}(t) \in \script{X}$, and: 
\begin{eqnarray*}
f(\overline{x}(t)) &\leq& \overline{y}_0(t) \\
g_k(\overline{x}(t)) &\leq& \overline{y}_k(t) \: \: \forall k \in \{1, \ldots, K\} 
\end{eqnarray*}
\end{lem} 

\begin{proof} 
Part (a) follows immediately from  convexity of $\script{X}$ and Jensen's inequality on the convex functions $f(x)$ and $g_k(x)$.  Part (b) follows by taking expectations of the inequalities in part (a) and again using Jensen's inequality.  Formally, it also uses the fact that if $X$ is a random vector that takes values in a convex set $\script{X}$, and if $\expect{X}$ is finite, then $\expect{X} \in \script{X}$. 
\end{proof} 

\begin{lem} If $x^*$ is an optimal solution to the convex program \eqref{eq:cp1}-\eqref{eq:cp3}, then $x(t)=x^*$ for all $t \in \{0, 1, 2, \ldots\}$  is an optimal solution to \eqref{eq:trans1}-\eqref{eq:trans3}.  Further, the optimal objective function value in both problems \eqref{eq:cp1}-\eqref{eq:cp3} and \eqref{eq:trans1}-\eqref{eq:trans3} is $f^*$. 
\end{lem} 

\begin{proof} 
Recall that $f^*$ is defined as the optimal objective function value for \eqref{eq:cp1}-\eqref{eq:cp3}. 
Let $x^*$ be an optimal solution to \eqref{eq:cp1}-\eqref{eq:cp3}, so that $x^*\in \script{X}$, $g_k(x^*) \leq c_k$ for all $k \in \{1, \ldots, K\}$, and $f(x^*)=f^*$.  Define $x(t) = x^*$ for all $t$.  Then \eqref{eq:trans3} clearly holds.  Further, for all $t>0$ one has: 
\begin{eqnarray*}
\overline{y}_k(t) = \frac{1}{t}\sum_{\tau=0}^{t-1} g_k(x^*) = g_k(x^*) \leq c_k \: \: \forall k \in \{1, \ldots, K\} 
\end{eqnarray*}
and so the constraints \eqref{eq:trans2} hold.  Similarly, $\overline{y}_0(t) = f(x^*)=f^*$ for all $t$.  Thus,  $x(t)$ satisfies the constraints of problem \eqref{eq:trans1}-\eqref{eq:trans3} and gives an objective function value of $f^*$. It follows that $f^* \geq y^*_0$, where $y^*_0$ is defined as the infimum objective function value over all $x(t)$ functions that meet the constraints of problem \eqref{eq:trans1}-\eqref{eq:trans3}. 

It remains to show that $f^* \leq y^*_0$ (so that $f^*=y^*_0$). 
To this end, let $x(t)$ be any (possibly random) process that satisfies the constraints of problem \eqref{eq:trans1}-\eqref{eq:trans3}. Since $x(t) \in \script{X}$ for all $t$, it follows that $\overline{x}(t) \in \script{X}$ for all $t$. Since $\script{X}$ is compact, the Bolzano-Wierstrass theorem implies there is a subsequence of times $t_m$ that increase to infinity such that:  
\begin{equation} \label{eq:limit} 
 \lim_{m\rightarrow\infty} \overline{x}(t_m) = \hat{x} 
 \end{equation} 
for some fixed vector $\hat{x}=(\hat{x}_1, \ldots, \hat{x}_N)  \in \script{X}$. 
Furthermore, Jensen's inequality (specifically, Lemma \ref{lem:jensen}b)  implies that for any time $t_m>0$: 
\begin{eqnarray}
f(\overline{x}(t_m)) &\leq& \overline{y}_0(t_m)  \label{eq:jensen2}  \\
g_k(\overline{x}(t_m)) &\leq& \overline{y}_k(t_m)  \: \: \forall k \in \{1, \ldots, K\}  \label{eq:jensen1}  
\end{eqnarray}
Therefore, by continuity of $g_k(x)$:  
\begin{eqnarray}
g_k(\hat{x}) &=& \lim_{m\rightarrow\infty} g_k(\overline{x}(t_m))  \label{eq:first} \\
&\leq&  \lim_{m\rightarrow\infty} \overline{y}_k(t_m) \label{eq:second} \\
&\leq& \limsup_{t\rightarrow\infty} \overline{y}_k(t) \nonumber \\
&\leq& c_k \label{eq:last} 
 \end{eqnarray}
 where \eqref{eq:first} holds by \eqref{eq:limit}, \eqref{eq:second} holds by \eqref{eq:jensen1}, and 
 \eqref{eq:last} holds because \eqref{eq:trans2} is satisfied.  Thus, $\hat{x}$ satisfies the constraints of problem \eqref{eq:cp1}-\eqref{eq:cp3}.  It follows that $f(\hat{x}) \geq f^*$, and so: 
  \begin{eqnarray}
 f^* &\leq&   f(\hat{x}) \nonumber\\
 &=& \lim_{m\rightarrow\infty} f(\overline{x}(t_m)) \label{eq:continuity} \\
  &\leq& \lim_{m\rightarrow\infty} \overline{y}_0(t_m) \label{eq:first-ineq} \\
  &\leq& \limsup_{t\rightarrow\infty} \overline{y}_0(t) \nonumber 
  \end{eqnarray}
  where \eqref{eq:continuity} holds by 
\eqref{eq:limit} and continuity of $f(x)$, and   
  \eqref{eq:first-ineq} holds by \eqref{eq:jensen2}. 
 Thus: 
 \[ f^* \leq \limsup_{t\rightarrow\infty} \overline{y}_0(t) \]
 This says that $f^*$ is less than or equal to the objective function value for any random process $x(t)$ that satisfies the constraints of problem \eqref{eq:trans1}-\eqref{eq:trans3}.  It follows that $f^* \leq y^*_0$. 
 \end{proof} 
 
 \subsection{Drift-plus-penalty for convex programs} 

The drift-plus-penalty algorithm to solve \eqref{eq:trans1}-\eqref{eq:trans3} defines virtual queues $Q_k(t)$ for $k \in \{1, \ldots, K\}$ by: 
\[ Q_k(t+1) = \max[Q_k(t) + y_k(t) - c_k, 0] \] 
Since $y_k(t) = g_k(x(t))$, this is equivalent to: 
\begin{equation} \label{eq:cp-q} 
Q_k(t+1) = \max[Q_k(t) + g_k(x(t)) - c_k, 0] 
\end{equation} 
The queues are initialized to zero.  Then every slot $t \in \{0, 1, 2, \ldots\}$: 
\begin{itemize} 
\item Observe $(Q_1(t), \ldots, Q_K(t))$ and choose $x(t) \in \script{X}$ to minimize: 
\begin{equation} \label{eq:cp-to-min} 
 Vf(x(t)) + \sum_{k=1}^KQ_k(t) g_k(x(t)) 
 \end{equation} 
\item Update $Q_k(t)$ via \eqref{eq:cp-q} for each $k \in \{1, \ldots, K\}$. 
\end{itemize} 
 
 Fix $\epsilon>0$. 
The next subsection shows that by defining $V=1/\epsilon$, the average of values 
$\frac{1}{t}\sum_{\tau=0}^{t-1} x(\tau)$ obtained from the above algorithm converges to an $O(\epsilon)$ approximation of \eqref{eq:cp1}-\eqref{eq:cp3} with convergence time $O(1/\epsilon^2)$.  The above drift-plus-penalty algorithm in this special case of a (deterministic) convex program is similar to the basic  \emph{dual subgradient algorithm} with step size $1/V$ (see, for example, \cite{bertsekas-convex}). However, a traditional analysis of the dual subgradient algorithm relies on \emph{strict convexity} assumptions to ensure that the primal values $x(t)$ converge to a $O(\epsilon)$-approximation of a (unique) optimal solution $x^*$.  The above requires only convexity (not strict convexity), and so there may be more than one optimal solution to \eqref{eq:cp1}-\eqref{eq:cp3}.  It then takes a time average of the primals to obtain an $O(\epsilon)$-approximation. 

 \subsection{Convex progam performance} 

There is no random event process $\omega(t)$ for this convex programming problem, and so the drift-plus-penalty algorithm makes purely deterministic decisions to minimize \eqref{eq:cp-to-min} every slot $t$.  Indeed, assume that if there are ties in the decision \eqref{eq:cp-to-min}, the tie is broken using some deterministic method.  The resulting sequence $\{x(t)\}_{t=0}^{\infty}$ is deterministic. 
It follows that all expectations in the analysis of the previous section can be removed.\footnote{Alternatively, one can repeat the same analysis of the previous section in the special case of no randomness, redefining $\overline{x}(t)$ and $\overline{y}_k(t)$ to be pure time averages without an expectation,  to obtain the same results for this deterministic convex program.} 
Thus, for all $t>0$: 
\begin{eqnarray*}
\overline{y}(t) &=& \frac{1}{t}\sum_{\tau=0}^{t-1}y(\tau) \\
\overline{x}(t) &=& \frac{1}{t}\sum_{\tau=0}^{t-1} x(\tau)
\end{eqnarray*}

For this convex programming problem, the Lagrange multiplier condition \eqref{eq:exist-lm} reduces to the existence of a vector $(\mu_1, \ldots, \mu_K)$ with non-negative components such that: 
\[ f(x) + \sum_{k=1}^K \mu_k g_k(x) \geq f(x^*) + \sum_{k=1}^K\mu_k c_k \: \: \forall x \in \script{X} \]

Fix $\epsilon>0$. 
It follows by Theorem \ref{thm:performance} that if the problem is feasible and has a Lagrange multiplier vector, then the drift-plus-penalty method with $V=1/\epsilon$ yields the following for all $t \geq 1/\epsilon^2$: 
\begin{eqnarray*}
\overline{y}_0(t) &\leq& f^* + O(\epsilon) \\
\overline{y}_k(t) &\leq& c_k + O(\epsilon) \: \: \forall k \in \{1, \ldots, K\} 
\end{eqnarray*}
On the other hand, it is clear by Lemma \ref{lem:jensen} (Jensen's inequality) that for all $t>0$: 
\begin{eqnarray*}
 f(\overline{x}(t)) &\leq&\overline{y}_0(t)  \\
g_k(\overline{x}(t)) &\leq& \overline{y}_k(t) \: \: \forall k \in \{1, \ldots, K\} 
\end{eqnarray*}
and hence $\overline{x}(t) \in \script{X}$ for all $t>0$, and: 
\begin{eqnarray*}
f(\overline{x}(t))  &\leq& f^* + O(\epsilon) \\
g_k(\overline{x}(t)) &\leq& c_k + O(\epsilon) \: \: \forall k \in \{1, \ldots, K\} 
\end{eqnarray*}
Thus, the drift-plus-penalty algorithm produces an $O(\epsilon)$ approximation to the convex program with convergence time $O(1/\epsilon^2)$. 

\subsection{Application to linear programs} 

Consider the special case of a \emph{linear program}, so that the $f(x)$ and $g_k(x)$ functions are linear and the set $\script{X}$ is replaced by a hyper-rectangle: 
\begin{eqnarray*}
\mbox{Minimize:} & \sum_{i=1}^N b_ix_i \\
\mbox{Subject to:} & \sum_{i=1}^N a_{ki} x_i \leq c_k \: \: \forall k \in \{1, \ldots, K\} \\
& x_{i,min}\leq x_i \leq x_{i,max} \: \: \forall i \in \{1, \ldots, N\} 
\end{eqnarray*}
where $x_{i,min}, x_{i,max}$, $b_i$, $a_{ki}$, and $c_k$ are given real numbers for all $i\in\{1, \ldots, N\}$ and $k \in \{1, \ldots, K\}$.  It is assumed that $x_{i,min} < x_{i,max}$ for all $i \in \{1, \ldots, N\}$. 
This fits the form of the convex program \eqref{eq:cp1}-\eqref{eq:cp3} via: 
\begin{eqnarray*}
f(x) &=& \sum_{i=1}^Nb_ix_i \\
g_k(x) &=& \sum_{i=1}^N a_{ki}x_i \\
\script{X} &=& \{x \in \mathbb{R}^N | x_{i,min} \leq x_i \leq x_{i,max} \: \: \forall i \in \{1, \ldots, N\} \} 
\end{eqnarray*}
The resulting drift-plus-penalty algorithm defines virtual queues: 
\begin{equation} \label{eq:q-update-linear} 
Q_k(t+1) = \max\left[Q_k(t) + \sum_{i=1}^N a_{ki}x_i(t) - c_k, 0\right] 
\end{equation} 
The queues are initialized to $0$.  Then every slot $t \in \{0, 1, 2, \ldots\}$, a vector $x(t) \in \script{X}$ is chosen to minimize: 
\[ V\sum_{i=1}^N b_i x_i + \sum_{k=1}^K Q_k(t)\left[\sum_{i=1}^Na_{ki}x_i(t)\right]  \]
This results in the following simple and separable optimization over each variable $x_i(t)$.  Every slot $t \in \{0, 1, 2, \ldots\}$: 
\begin{itemize}
\item Observe $Q_1(t), \ldots, Q_K(t)$.  For each $i \in \{1, \ldots, N\}$ choose: 
\[ x_i(t) = \left\{ \begin{array}{ll}
x_{i,max} &\mbox{ if $Vb_i + \sum_{k=1}^KQ_k(t)a_{ki} \leq 0$} \\
x_{i,min} & \mbox{ otherwise} 
\end{array}
\right. \]
\item Update $Q_k(t)$ for $k \in \{1, \ldots, K\}$ via \eqref{eq:q-update-linear}. 
\item Update $\overline{x}(t)$ via $\overline{x}(t+1) = \overline{x}(t)\frac{t}{t+1} + \frac{x(t)}{t+1}$. 
\end{itemize}

This algorithm always chooses $x_i(t)$ within the 2-element set $\{x_{i,min}, x_{i,max}\}$.  Thus, the $x(t)$ vectors themselves cannot converge to an approximate solution if the resulting solution is not a corner point on the hyper-rectangle $\script{X}$ (for example, optimality might require $x_1^* = (x_{1,min} + x_{1,max})/2$).  However, Theorem \ref{thm:performance} ensures the \emph{time averages} $\overline{x}(t)$ converge to 
an $O(\epsilon)$-approximation with convergence time $O(1/\epsilon^2)$. 

\section{Distributed optimization over a connected graph} 

Consider a directed graph with $N$ nodes.  Let $\script{N} = \{1, \ldots, N\}$ be the set of nodes.  Let $\script{L}$ be the set of all directed links.  Each node $n \in \script{N}$ has a vector of its own variables $x^{(n)} = (x_1^{(n)}, \ldots, x_{M_n}^{(n)}) \in \mathbb{R}^{M_n}$, where $M_n$ is a positive integer for each $n \in \script{N}$. In addition, there is a vector $\theta = (\theta_1, \ldots, \theta_G) \in \mathbb{R}^G$ of \emph{common variables} (for some positive integer $G$).  The goal is to solve the problem in a distributed way, so that each node makes decisions based only on information available from its neighbors.  The problem and approach in this section is a variation on the work in \cite{neely-dist-comp}. 

Each node $n \in \script{N}$ must choose variables $x^{(n)} \in \script{X}^{(n)}$, where $\script{X}^{(n)}$ is a convex and compact subset of $\mathbb{R}^{M_n}$.  In addition, the nodes must collectively choose $\theta \in \Theta$, where $\Theta$ is a convex and compact subset of $\mathbb{R}^G$.  The goal is to solve: 
\begin{eqnarray}
\mbox{Minimize:} & \sum_{n=1}^N f^{(n)}(x^{(n)}, \theta)  \label{eq:dist1} \\
\mbox{Subject to:} & g^{(n)}(x^{(n)}, \theta) \leq c^{(n)} \: \: \forall n \in \script{N} \label{eq:dist2}  \\
& x^{(n)} \in \script{X}^{(n)} \: \: \forall n \in \script{N} \label{eq:dist3} \\
& \theta \in \Theta \label{eq:dist4} 
\end{eqnarray}
where $f^{(n)}(x^{(n)}, \theta)$ and $g^{(n)}(x^{(n)}, \theta)$ are convex functions over $\script{X}^{(n)} \times \Theta$, defined for each $n \in \script{N}$.  

The goal is to solve this problem by making distributed decisions at each node.  The difficulty is that 
the $\theta$ variables must be chosen collectively.  The next subsection clarifies the challenges by specifying the drift-plus-penalty algorithm.  Subsection \ref{section:dist-implementation1} modifies the problem (without affecting optimality) to produce a distributed solution.

\subsection{The direct drift-plus-penalty approach} 

The problem \eqref{eq:dist1}-\eqref{eq:dist4} is a convex program.  The drift-plus-penalty method defines virtual queues $Q^{(n)}(t)$ for each $n \in \script{N}$ to enforce the constraints \eqref{eq:dist2}: 
\begin{equation} \label{eq:dist-q-update} 
Q^{(n)}(t+1) = \max[Q^{(n)}(t) + g^{(n)}(x^{(n)}(t), \theta(t)) - c^{(n)}, 0] \: \: \forall n \in \script{N} 
\end{equation} 
Every slot $t \in \{0, 1, 2, \ldots\}$, the algorithm chooses $x^{(n)}(t) \in \script{X}^{(n)}$ for all $n \in \script{N}$, and chooses $\theta(t) \in \Theta$ to minimize: 
\[ \sum_{n=1}^N Vf^{(n)}(x^{(n)}(t), \theta(t)) + \sum_{n=1}^N Q^{(n)}(t)g^{(n)}(x^{(n)}(t), \theta(t)) \]
The difficulty is the joint selection of the $\theta(t)$ variables, which couples all terms together in a centralized optimization. 

\subsection{A distributed approach}\label{section:dist-implementation1} 

This subsection specifies a distributed solution, along the lines of the general solution methodology from \cite{neely-dist-comp}.  The idea is to introduce \emph{estimation vectors} $\theta^{(n)}(t)\in\Theta$ at each node $n \in \script{N}$.  
Consider the following problem: 
\begin{eqnarray}
\mbox{Minimize:} & \sum_{n=1}^N f^{(n)}(x^{(n)}, \theta^{(n)}) \label{eq:new1} \\
\mbox{Subject to:} & g^{(n)}(x^{(n)}, \theta^{(n)}) \leq c^{(n)} \: \: \forall n \in \script{N} \label{eq:new2} \\
& \theta^{(n)} =  \theta^{(j)}  \: \: \forall (n,j) \in \script{L} \label{eq:new3} \\
& x^{(n)} \in \script{X}^{(n)} \: \: \forall n \in \script{N} \label{eq:new4} \\
& \theta^{(n)} \in \Theta \: \: \forall n \in \script{N} \label{eq:new5} 
\end{eqnarray}
The constraints \eqref{eq:new3} are \emph{vector equality} constraints.  Specifically, if $\theta^{(n)} = (\theta^{(n)}_1, \ldots, \theta^{(n)}_G)$, then the constraints are: 
\begin{equation} \label{eq:dude-yes} 
\theta_i^{(n)} = \theta_i^{(j)} \: \: \forall i \in \{1, \ldots, G\}, \forall (n,j) \in \script{L} 
\end{equation} 

Now assume that if one changes the directed graph to an undirected graph by changing all directed links to undirected links, then the resulting undirected graph is connected (so that there is a path from every node to every other node in the undirected graph).   With this connectedness assumption, the problem \eqref{eq:new1}-\eqref{eq:new5} is equivalent to the original problem \eqref{eq:dist1}-\eqref{eq:dist4}. 
That is because for any nodes $n$ and $m$ in $\script{N}$, there is a path in the undirected graph from $n$ to $m$, and the equality constraints \eqref{eq:new3} ensure that each node $j$ on this path has  $\theta^{(j)}=\theta^{(n)}$.  It follows that the constraints \eqref{eq:new3} ensure that  the estimation 
vectors $\theta^{(n)}$ are the same for all nodes $n \in \script{N}$. 

The problem \eqref{eq:new1}-\eqref{eq:new5} can be solved via the drift-plus-penalty framework of Section \ref{section:equality}.  For each inequality constraint \eqref{eq:new2} (that is, for each $n \in \script{N}$), define: 
\begin{equation} \label{eq:graph-dist-q-update} 
 Q^{(n)}(t+1) = \max[Q^{(n)}(t) + g^{(n)}(x^{(n)}(t), \theta^{(n)}(t)) - c^{(n)}, 0] 
 \end{equation} 
For each equality constraint \eqref{eq:dude-yes} (that is, for each $i \in \{1, \ldots, G\}$ and $(n,j)$ in $\script{L}$) define: 
\begin{equation} \label{eq:graph-dist-z-update} 
Z^{(n,j)}_i(t+1) = Z^{(n,j)}_i(t) + \theta^{(n)}_i(t) - \theta^{(j)}_i(t) 
\end{equation} 
Each node $n \in \script{N}$ is responsible for updating queues 
$Q^{(n)}(t)$ and $Z_i^{(n,j)}(t)$ for all $i \in \{1, \ldots, G\}$ and all  $j$ such that $(n,j) \in \script{L}$. Every slot $t$, decisions are made to minimize: 
\begin{align*}
&\sum_{n=1}^NVf^{(n)}(x^{(n)}(t), \theta^{(n)}(t)) \\
&+ \sum_{n\in\script{N}} Q^{(n)}(t)g^{(n)}(x^{(n)}(t), \theta^{(n)}(t))\\
&+ \sum_{i=1}^G\sum_{(n,j) \in \script{L}} Z^{(n,j)}_i(t)(\theta^{(n)}_i(t) - \theta^{(j)}_i(t)) 
\end{align*}

This is a separable optimization in each of the local variables $x^{(n)}(t)$ and $\theta^{(n)}(t)$ associated with individual nodes $n \in \script{N}$.  Each node $n \in \script{N}$ needs to know only its own internal queues and the queue values $Z_i^{(a,n)}(t)$ of its neighbors.  It is assumed that these values can be obtained via message passing on the links associated with each neighbor. 
The resulting algorithm is as follows:  Initialize all queues to $0$. 
Every slot $t \in \{0, 1, 2, \ldots\}$ do: 

\begin{itemize} 
\item Each node $n \in \script{N}$ observes $Q^{(n)}(t)$ and the queues $Z^{(n,j)}_i(t)$ and $Z^{(a,n)}_i(t)$ for all $(n,j) \in \script{L}$ and all $(a,n) \in \script{L}$, and all $i \in \{1, \ldots, G\}$. It then chooses $(x^{(n)}(t), \theta^{(n)}(t)) \in \script{X}^{(n)}\times \Theta$ to minimize: 
\begin{align*}
&Vf^{(n)}(x^{(n)}(t), \theta^{(n)}(t))  + Q^{(n)}(t)g^{(n)}(x^{(n)}(t), \theta^{(n)}(t))  \\
& + \sum_{i=1}^G\theta^{(n)}_i(t)\left[ \sum_{j | (n,j) \in \script{L}} Z^{(n,j)}_i(t)  - \sum_{a | (a,n) \in \script{L}} Z^{(a,n)}_i(t) \right]
\end{align*} 

\item Each node $n \in \script{N}$ updates $Q^{(n)}(t)$ via \eqref{eq:graph-dist-q-update} and 
updates $Z^{(n,j)}_i(t)$ for $(n,j) \in \script{L}$ via \eqref{eq:graph-dist-z-update}. The $Z^{(n,j)}_i(t)$ update for node $n$ requires all neighbors $j$ such that $(n,j) \in \script{L}$ to first pass their chosen $\theta^{(j)}(t)$ vectors to node $n$, so that the right-hand-side of \eqref{eq:graph-dist-z-update} can be computed. 
\end{itemize} 

Fix $\epsilon>0$. 
Using $V = 1/\epsilon$, the resulting time averages $\overline{x}^{(n)}(t)$ and $\overline{\theta}^{(n)}(t)$ 
converge to an $O(\epsilon)$ approximation with convergence time $O(1/\epsilon^2)$. 

\subsection{A different type of constraint}

The problem \eqref{eq:new1}-\eqref{eq:new5} specifies one constraint of the form  
$g^{(n)}(x^{(n)}, \theta) \leq c^{(n)}$ for each node $n \in \script{N}$.  Suppose the problem is changed so that these constraints \eqref{eq:new2} are replaced by a single constraint of the form: 
\begin{equation} \label{eq:hypothetical} 
 \sum_{n\in\script{N}} g^{(n)}(x^{(n)}, \theta^{(n)}) \leq c 
 \end{equation} 
 for some given real number $c$. 
In principle, this could be treated using a virtual queue: 
\[ J(t+1) = \max\left[J(t) + \sum_{n\in\script{N}} g^{(n)}(x^{(n)}(t), \theta^{(n)}(t)) - c, 0\right]  \]
However, it is not clear \emph{which node should implement this queue}.  Further, every slot $t$,  that node would need to know values of $g^{(n)}(x^{(n)}(t), \theta^{(n)}(t))$ for all nodes $n \in \script{N}$, which is difficult in a distributed context. 

One way to avoid this difficulty is as follows:  Form new variables $x^{(n,m)} \in \script{X}^{(n)}$ for all $n, m \in \script{N}$.  The variable $x^{(n,m)}$ can be interpreted as the node $m$ estimate of the optimal value of $x^{(n)}$.  The constraint \eqref{eq:hypothetical} is then replaced by: 
\begin{align}
&\sum_{n \in \script{N}} g^{(n)}(x^{(n,1)}, \theta^{(1)}) \leq c \label{eq:weird1}  \\
&x^{(n,m)} = x^{(n,j)} \: \: \forall n \in \script{N}, \forall (m,j) \in \script{L} \label{eq:weird2} \\
& x^{(n,m)} \in \script{X}^{(n)} \: \: \forall n \in \script{N} \label{eq:weird3}  
\end{align}
Node 1 is responsible for the constraint \eqref{eq:weird1} and maintains a virtual queue: 
\[ J(t+1) = \max\left[J(t) + \sum_{n\in\script{N}} g^{(n)}(x^{(n,1)}(t), \theta^{(1)}(t)) - c, 0\right] \]
Each node $m \in \script{N}$ is responsible for the vector equality 
constraints $x^{(n,m)}=x^{(n,j)}$ for all $n \in \script{N}$ and all $(m,j) \in \script{L}$. These are enforced in the same manner as the constraints \eqref{eq:new3}. 

\section{Conclusions} 

This paper proves $O(1/\epsilon^2)$ convergence time for the drift-plus-penalty algorithm in a general situation where a Lagrange multiplier vector exists, without requiring a Slater condition.  This holds for both stochastic optimization problems and for (deterministic) convex programs.  Special case implementations were given for convex programs, including linear programs.  Example solutions were also presented for solving convex programs in a distributed way over a connected graph.

\bibliographystyle{unsrt}
\bibliography{../../latex-mit/bibliography/refs}

\begin{thebibliography}{10}

\bibitem{sno-text}
M.~J. Neely.
\newblock {\em Stochastic Network Optimization with Application to
  Communication and Queueing Systems}.
\newblock Morgan \& Claypool, 2010.

\bibitem{dist-opt-arxiv}
M.~J. Neely.
\newblock Distributed stochastic optimization via correlated scheduling.
\newblock {\em ArXiv technical report, arXiv:1304.7727v2}, May 2013.

\bibitem{neely-dist-comp}
M.~J. Neely.
\newblock Distributed and secure computation of convex programs over a network
  of connected processors.
\newblock {\em DCDIS Conf., Guelph, Ontario}, July 2005.

\bibitem{neely-fairness-ton}
M.~J. Neely, E.~Modiano, and C.~Li.
\newblock Fairness and optimal stochastic control for heterogeneous networks.
\newblock {\em IEEE/ACM Transactions on Networking}, vol. 16, no. 2, pp.
  396-409, April 2008.

\bibitem{neely-thesis}
M.~J. Neely.
\newblock {\em Dynamic Power Allocation and Routing for Satellite and Wireless
  Networks with Time Varying Channels}.
\newblock PhD thesis, Massachusetts Institute of Technology, LIDS, 2003.

\bibitem{now}
L.~Georgiadis, M.~J. Neely, and L.~Tassiulas.
\newblock Resource allocation and cross-layer control in wireless networks.
\newblock {\em Foundations and Trends in Networking}, vol. 1, no. 1, pp. 1-149,
  2006.

\bibitem{neely-energy-it}
M.~J. Neely.
\newblock Energy optimal control for time varying wireless networks.
\newblock {\em IEEE Transactions on Information Theory}, vol. 52, no. 7, pp.
  2915-2934, July 2006.

\bibitem{tass-radio-nets}
L.~Tassiulas and A.~Ephremides.
\newblock Stability properties of constrained queueing systems and scheduling
  policies for maximum throughput in multihop radio networks.
\newblock {\em IEEE Transacations on Automatic Control}, vol. 37, no. 12, pp.
  1936-1948, Dec. 1992.

\bibitem{tass-server-allocation}
L.~Tassiulas and A.~Ephremides.
\newblock Dynamic server allocation to parallel queues with randomly varying
  connectivity.
\newblock {\em IEEE Transactions on Information Theory}, vol. 39, no. 2, pp.
  466-478, March 1993.

\bibitem{atilla-fairness-ton}
A.~Eryilmaz and R.~Srikant.
\newblock Fair resource allocation in wireless networks using
  queue-length-based scheduling and congestion control.
\newblock {\em IEEE/ACM Transactions on Networking}, vol. 15, no. 6, pp.
  1333-1344, Dec. 2007.

\bibitem{longbo-lagrange-tac}
L.~Huang and M.~J. Neely.
\newblock Delay reduction via {L}agrange multipliers in stochastic network
  optimization.
\newblock {\em IEEE Transactions on Automatic Control}, vol. 56, no. 4, pp.
  842-857, April 2011.

\bibitem{neely-energy-convergence-arxiv}
M.~J. Neely.
\newblock Energy-aware wireless scheduling with near optimal backlog and
  convergence time tradeoffs.
\newblock {\em ArXiv technical report, arXiv:1411.4740}, Nov. 2014.

\bibitem{sucha-convergence-time}
S.~Supittayapornpong, L.~Huang, and M.~J. Neely.
\newblock Time-average optimization with nonconvex decision set and its
  convergence.
\newblock In {\em Proc. IEEE Conf. on Decision and Control (CDC)}, Los Angeles,
  California, Dec. 2014.

\bibitem{wei-convergence-admm}
E.~Wei and A.~Ozdaglar.
\newblock On the ${O}(1/k)$ convergence of asynchronous distributed alternating
  direction method of multipliers.
\newblock In {\em Proc. IEEE Global Conference on Signal and Information
  Processing}, 2013.

\bibitem{bertsekas-convex}
D.~P. Bertsekas, A.~Nedic, and A.~E. Ozdaglar.
\newblock {\em Convex Analysis and Optimization}.
\newblock Boston: Athena Scientific, 2003.

\bibitem{lin-shroff-cdc04}
X.~Lin and N.~B. Shroff.
\newblock Joint rate control and scheduling in multihop wireless networks.
\newblock {\em Proc. of 43rd IEEE Conf. on Decision and Control, Paradise
  Island, Bahamas}, Dec. 2004.

\bibitem{lee-stochastic-scheduling}
J.~W. Lee, R.~R. Mazumdar, and N.~B. Shroff.
\newblock Opportunistic power scheduling for dynamic multiserver wireless
  systems.
\newblock {\em IEEE Transactions on Wireless Communications}, vol. 5, no.6, pp.
  1506-1515, June 2006.

\bibitem{prop-fair-down}
H.~Kushner and P.~Whiting.
\newblock Asymptotic properties of proportional-fair sharing algorithms.
\newblock {\em Proc. 40th Annual Allerton Conf. on Communication, Control, and
  Computing, Monticello, IL}, Oct. 2002.

\bibitem{vijay-allerton02}
R.~Agrawal and V.~Subramanian.
\newblock Optimality of certain channel aware scheduling policies.
\newblock {\em Proc. 40th Annual Allerton Conf. on Communication, Control, and
  Computing, Monticello, IL}, Oct. 2002.

\bibitem{stolyar-greedy}
A.~Stolyar.
\newblock Maximizing queueing network utility subject to stability: Greedy
  primal-dual algorithm.
\newblock {\em Queueing Systems}, vol. 50, no. 4, pp. 401-457, 2005.

\bibitem{stolyar-gpd-gen}
A.~Stolyar.
\newblock Greedy primal-dual algorithm for dynamic resource allocation in
  complex networks.
\newblock {\em Queueing Systems}, vol. 54, no. 3, pp. 203-220, 2006.

\bibitem{atilla-primal-dual-jsac}
A.~Eryilmaz and R.~Srikant.
\newblock Joint congestion control, routing, and {MAC} for stability and
  fairness in wireless networks.
\newblock {\em IEEE Journal on Selected Areas in Communications, Special Issue
  on Nonlinear Optimization of Communication Systems}, vol. 14, pp. 1514-1524,
  Aug. 2006.

\end{thebibliography}
\end{document}